\documentclass[12pt, a4paper, reqno, oneside]{amsart}
\usepackage{amsmath,amsthm,amssymb,mathrsfs}
\usepackage{color,hyperref,graphicx,comment}
\usepackage[utf8]{inputenc}
\usepackage[all]{xy}
\usepackage{mathtools}
\usepackage{enumitem}
\usepackage{faktor}
\usepackage{tikz-cd}
\usepackage[all]{xy} 
\usepackage{upgreek}
\usepackage{bbold} 
\usepackage{stmaryrd} 

\usepackage{geometry}
\numberwithin{equation}{section}

\theoremstyle{plain}
  \newtheorem{thm}{Theorem}[section]
  \newtheorem{lem}[thm]{Lemma}
  \newtheorem{prop}[thm]{Proposition}
  
  \newtheorem*{thm*}{Theorem}  
  \newtheorem*{thmm*}{Main Theorem} 
\theoremstyle{definition}
  \newtheorem{defn}[thm]{Definition}
  \newtheorem{rmk}[thm]{Remark}
  \newtheorem*{ack*}{Acknowledgement}
  \newtheorem*{ref*}{Reference}
  \newtheorem*{ex*}{Example}
  \newtheorem*{ft*}{Fact}
  \newtheorem*{dis*}{Disclosure of Resources}
\theoremstyle{plain}

\newcommand\pl{\partial}
\newcommand\iip[2]{\left\langle{#1},{#2}\right\rangle}

\newcommand\af{\alpha}
\newcommand\bt{\beta}

\newcommand\gm{\gamma}
\newcommand\Om{\Omega}

\newcommand\sm{\sigma}

\newcommand\dt{\delta}
\newcommand\Dt{\Delta}
\newcommand\vep{\varepsilon}
\newcommand\vph{\varphi}

\newcommand\BR{\mathbb{R}}

\newcommand\BZ{\mathbb{Z}}
\newcommand\BN{\mathbb{N}}

\newcommand\BS{\mathbb{S}}

\newcommand\bv{\mathbf{v}}
\newcommand\bw{\mathbf{w}}

\newcommand\bB{\mathbf{B}}
\newcommand\bY{\mathbf{Y}}
\newcommand\bn{\mathbf{n}}
\newcommand\bM{\mathbf{M}}

\newcommand\CA{\mathcal{A}}
\newcommand\CC{\mathcal{C}}

\newcommand\CH{\mathcal{H}}

\newcommand\SA{\mathscr{A}}
\newcommand\SB{\mathscr{B}}

\newcommand\SL{\mathscr{L}}

\DeclareMathOperator{\tr}{tr}

\DeclareMathOperator{\Int}{int}
\DeclareMathOperator{\spn}{span}

\title{Perturbations of a minimal surface with triple junctions in $\BR^2 \times \BS^1$}
\author{Chen-Kuan Lee}
\address{Department of Mathematics\\University of Notre Dame\\Notre Dame, IN 46556\\USA}
\email{clee36@nd.edu}

\begin{document}
\begin{sloppypar}

\begin{abstract}
We construct stationary perturbations of a specific minimal surface with a circle of triple junctions in $\BR^2 \times \BS^1$, that satisfy given boundary data.
\end{abstract}

\maketitle


\section{\bf Introduction}

Surfaces with triple junctions arise naturally in the study of minimal surfaces and soap bubble clusters. Taylor \cite{Tay76} showed that there are only two possible singularities of certain almost-area-minimizing surfaces ($(\mathbf{M}, \epsilon, \delta)$-minimal sets introduced by Almgren \cite{Alm76}) in $\BR^3$. One of them is the triple junction, which consists of three surfaces that meet at $120^\circ$, and the other is a tetrahedral-type junction, which consists of four triple junctions meeting at a point such that four spines pairwise meet at the same angle $arccos(-\frac{1}{3}) \approx 109^\circ$. Later, Lawlor and Morgan \cite{LM96} proved that submanifolds with triple junctions are always locally area-minimizing in arbitrary dimension and codimension.

In some ways, triple junctions are one of the simplest and nicest singular models. They are topologically stable, in the sense that minimal surfaces near triple junctions will also have triple junctions \cite{Sim93}, and they are also very regular: minimal surfaces with triple junctions are analytic up through their singular curve \cite{KNS78}, \cite{Kru14}. Consequently, triple junctions have been studied extensively in the contexts of minimal surfaces and mean curvature flow; see e.g. \cite{MNT04, TW16, INS19, SW20, BM21, Wan22, Wan24}.

In this paper, we construct examples of minimal surfaces with a circle of triple junctions modeled on $\bY \times \BS^1$ in $\BR^2 \times \BS^1$. Given any normal perturbation of the boundary circles $(\bY \times \BS^1) \cap \partial \bB$, where $\bB = B_1^2(0) \times \BS^1$ is the solid torus in $\BR^2 \times \BS^1$, we shall find perturbations of $(\bY \times \BS^1) \cap \bB$ which are stationary and have the prescribed boundary. We describe our precise setup and the main theorem below.

\subsection{Notation and the main result}

Consider the Euclidean space $(\BR^2, \iip{\cdot}{\cdot})$ and identify $\BS^1$ with $ \BR / \BZ$ inheriting the Riemannian metric from the Euclidean one. Then $\BR^2 \times \BS^1$ with the product metric is also a complete Riemannian manifold. Here we slightly abuse the notation: still denote the product metric by $\iip{\cdot}{\cdot}$. Let $\bY \subset \BR^2$ be the cone consisting of three rays that meet at $120^\circ$ (see Figure \ref{fig_Y} in Section 2): \begin{equation}\label{eqn:Y-def}
    \bY = \{(x, 0) \,\big|\, x \geq 0\} \cup \{(-x, \sqrt{3} x) \,\big|\, x \geq 0\} \cup \{(-x, - \sqrt{3} x) \,\big|\, x \geq 0\}.
\end{equation} We note that $\bY$ is a stationary rectifiable varifold in $\BR^2$, and hence $\bY \times \BS^1$, which contains a circle of triple junctions, is a $2$-dimensional stationary cone in $\BR^2 \times \BS^1$. In particular, $(\bY \times \BS^1) \cap \bB$ is stationary in $\bB$.

Throughout the paper, we often consider a triple of functions $\underline{u} \coloneqq (u_1, u_2, u_3): \Om \rightarrow \BR$, where $\Om$ is a manifold, possibly with boundary. Define $|\underline{u}| \coloneqq \sum_{i = 1}^3 |u_i|$ and let $\underline{\CC}^{k, \af}(\Om)$ be the triple of $\CC^{k, \af}(\Om)$ equipped with the norm $\|\underline{u}\|_{k, \af} \coloneqq \sum_{i = 1}^3 \|u_i\|_{k, \af}$, where $\|u_i\|_{k, \af}$ is the usual H\"{o}lder norm of $u_i \in \CC^{k, \af}(\Om)$. Note that when $\Om$ is compact, both $\CC^{k, \af}(\Om)$ and $\underline{\CC}^{k, \af}(\Om)$ are Banach spaces.

Our main result is the following theorem, which states that there are $\CC^{2, \af}$-perturbations of $(\bY \times \BS^1) \cap \bB$ that are stationary and have the prescribed boundary data.

\begin{thm} \label{thm_main}
    Given any $\dt \in (0, \frac{1}{2})$ and $\af \in (0, 1)$, there exists $\vep > 0$ satisfying the following: for all $\underline{\vph} \in \underline{\CC}^{2, \af}(\BS^1)$ with $\|\underline{\vph}\|_{2, \af} < \vep$, there exists $\underline{u} \in \underline{\CC}^{2, \af}([0, 1] \times \BS^1)$ such that \begin{enumerate}
        \item $\underline{u}\big|_{\{1\} \times \BS^1} = \underline{\vph}$.
        \item There is a stationary $\CC^{2, \af}$-perturbation of $(\bY \times \BS^1) \cap \bB$ associated with $(\underline{u}, \dt)$ (See Definition \ref{pert}).
    \end{enumerate}
\end{thm}

\begin{rmk}
    The perturbation associated with $(\underline{u}, \dt)$ is apriori only $\CC^{2, \af}$. However, by \cite{KNS78}, \cite{Kru14} and the classical bootstrapping argument, it is in fact analytic up to the triple junctions.
\end{rmk}

\begin{rmk}
    When $\dt$ is sufficiently small, the perturbation associated with $(\underline{u}, \dt)$ is the unique $\CC^{2, \af}$-perturbation with the prescribed boundary data $\underline{\vph}$. More precisely, when $\dt < 10 \Tilde{r}$ (See Theorem \ref{thm_nonlin}), the contraction mapping argument also yields the uniqueness of $\underline{u}$ in the ball of radius $\frac{\dt}{10}$ in $\underline{\CC}^{2, \af}([0, 1] \times \BS^1)$, from which the uniqueness of the perturbation follows.
\end{rmk}

Mese and Yamada \cite{MY06} used energy methods to find area-minimizing surfaces in Euclidean space with triple junctions bounding wire frames which are deformations of $(\bY \times \BR) \cap \partial B_1^3$, provided that they satisfy a certain non-degeneracy condition. Since we consider perturbations of $(\bY \times \BS^1) \cap \bB$, we avoid any degeneracy issues with the interior singular set interacting with the boundary, or with singularities in the boundary wire frame itself. 

We take a non-parametric approach in this paper, by explicitly describing a minimal surface with triple junctions as a kind of graph over the model case $\bY \times \BS^1$, and expressing the corresponding minimal surface equations as a perturbation of a linear system. We solve this system with (small) prescribed boundary data using a contraction mapping argument, and obtain quantitative estimates for our surface. We note that since the triple junction behaves like a free boundary in the surface, we cannot prescribe both the singular set and the boundary curves---instead, the singular curve is obtained as part of our solution.

Let us also remark that, even in general, there are relatively few concrete examples of minimal surfaces with non-trivial singular sets, with most constructions being fairly recent; see e.g. \cite{Sim23, Liu21, Liu23, Liu24}.

The organization of this paper is as follows. In Section 2, we elaborate on the definition of $\CC^{2, \af}$-perturbations and give a sufficient condition to be a stationary $\CC^{2, \af}$-perturbation. This leads to dealing with a system of quasilinear PDEs. Section 3 contains the solvability, uniqueness, and a Schauder estimate to the linearization of the PDE system. Finally, we use the contraction mapping to solve the original PDE system and prove the main theorem in Section 4.

\begin{ack*} The author is really grateful to Prof. Nick Edelen for his constant support and many inspiring comments. He would also like to thank Jui-Yun Hung for several valuable discussions. Finally, he thanks the anonymous referees for their helpful comments and suggestions.
\end{ack*}

\begin{dis*}
    The author used MathSciNet and arXiv for the literature review. All calculations in this paper were carried out by the author manually. No LLM tools were used in the preparation of this paper.
\end{dis*}

\section{\bf Parametrization of stationary varifolds as \texorpdfstring{$\CC^{2, \af}$}{TEXT}-perturbations of \texorpdfstring{$(\bY \times \BS^1) \cap \bB$}{TEXT}}

Let \begin{equation} \label{conormal}
    \bn_1 = (- 1, 0), \quad \bn_2 = (\frac{1}{2}, - \frac{\sqrt{3}}{2}), \quad \bn_3 = (\frac{1}{2}, \frac{\sqrt{3}}{2})
\end{equation} be three unit vectors in $\BR^2$. Rotating $\bn_i$ counterclockwise by $90^\circ$ to get $\nu_i$. That is, \begin{equation} \label{normal}
    \nu_1 = (0, -1), \quad \nu_2 = (\frac{\sqrt{3}}{2}, \frac{1}{2}), \quad \nu_3 = (-\frac{\sqrt{3}}{2}, \frac{1}{2}).
\end{equation}

\begin{figure}[h]
    \includegraphics[width=0.3\textwidth]{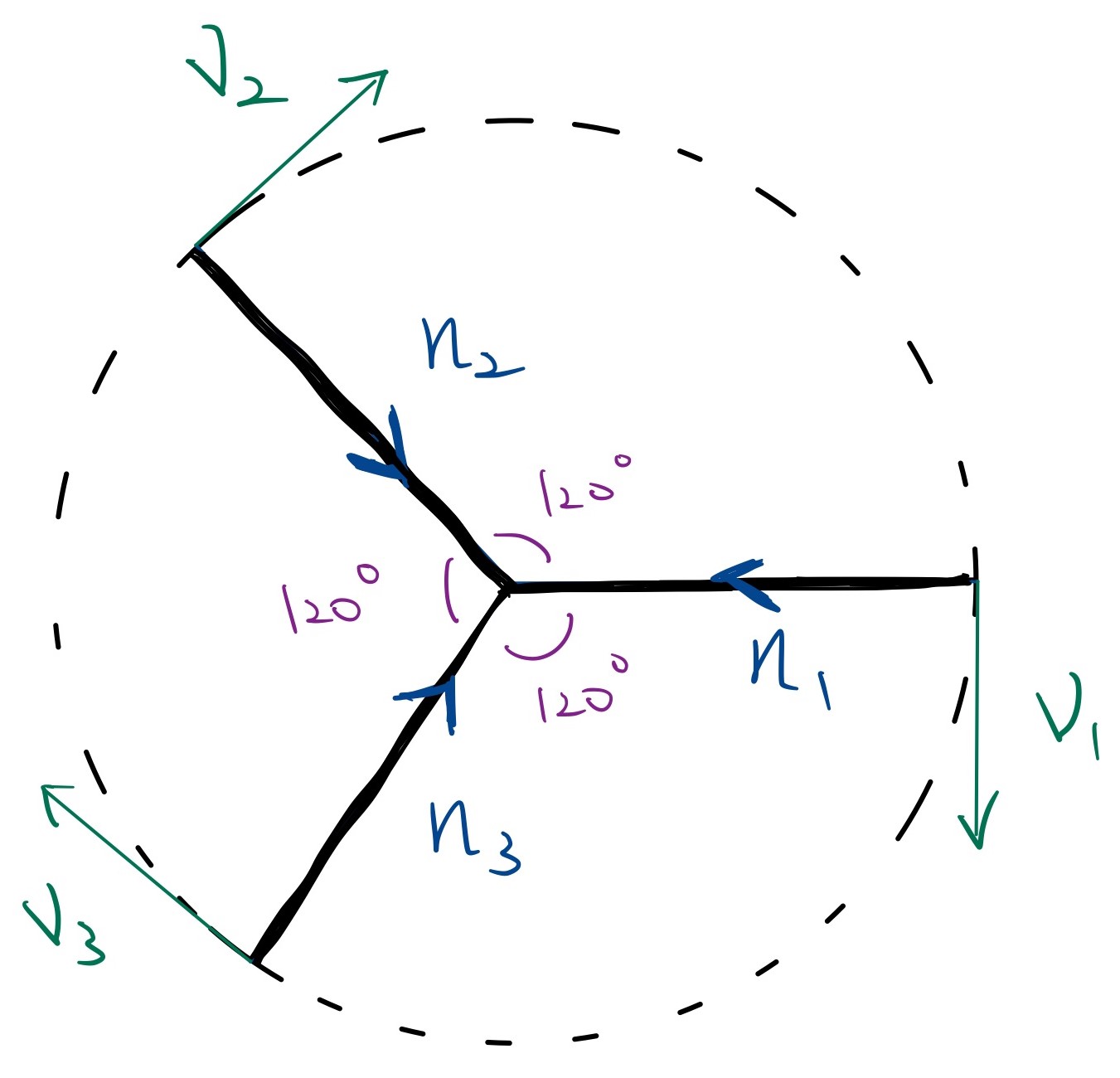}
    \caption{The graph of $\bY \cap B^2_1(0)$.}
    \label{fig_Y}
\end{figure}

For each $i = 1, 2, 3$, we define \begin{equation}
    \Om_i = \left\{ (-x \, \bn_i, y) \, \big| \, 0 \leq x < 1, y \in \BS^1 \right\} \subset \bB.
\end{equation} We then decompose $(\bY \times \BS^1) \cap \bB$ as the union of $\Om_1$, $\Om_2$ and $\Om_3$. $$\bigcap_{i = 1}^3 \Om_i = \left\{ ((0 , 0), y) \, \big| \, y \in \BS^1 \right\}$$ is called the spine of $(\bY \times \BS^1) \cap \bB$, which is an $\BS^1$ of triple junctions. After identifying the tangent space of $\BR^2 \times \BS^1$ with $\BR^3$, we note that $(0, 0, 1)$ is the unit tangent of the spine, $(\bn_i, 0)$ is the unit outer normal of the spine in $\Om_i$, and $(\nu_i, 0)$ is a unit normal of $\Om_i$ in $\bB$. Moreover, $\Om_i$ can be identified with $[0, 1) \times \BS^1$ in such a way that $\pl_{(\bn_i, 0)} = - \pl_x = \pl_\bn$, where $\bn$ is the unit outer normal of $\{0\} \times \BS^1$ in $[0, 1) \times \BS^1 \subset \BR \times \BS^1$.

Now, we want to perturb $(\bY \times \BS^1) \cap \bB$ along the normals. For fixed $\dt \in (0, \frac{1}{2})$, let $\eta: [0, 1] \rightarrow [0, 1]$ be a standard smooth cutoff function that is $1$ on $[0, \dt]$, $0$ on $[2 \dt, 1]$, and $|\eta'| \leq \frac{2}{\dt}$. Given $\underline{u} = (u_1, u_2, u_3): [0, 1) \times \BS^1 \rightarrow \BR$, we define $\bw_i: \BS^1 \rightarrow \BR^2$ by \begin{equation} \label{bw}
    \bw_i(y) = \frac{u_{i-1}(0, y) - u_{i + 1}(0, y)}{\sqrt{3}} \bn_i,
\end{equation} where $u_0 = u_3$, and parametrize \begin{equation}
    \bM = \bigcup_{i = 1}^3 \bM_i \coloneqq \bigcup_{i = 1}^3 \left\{ (- x\, \bn_i + u_i(x, y) \nu_i + \eta(x) \bw_i(y), y) \, \big| \, (x, y) \in [0, 1) \times \BS^1 \right\}.
\end{equation}

In this paper, by $\CC^{2, \af}$-perturbation we mean the following.

\begin{defn} \label{pert}
$\bM$ is called a $\CC^{2, \af}$-perturbation of $(\bY \times \BS^1) \cap \bB$ associated with $(\underline{u}, \dt)$ if \begin{enumerate}
    \item $\underline{u} \in \underline{\CC}^{2, \af}([0, 1] \times \BS^1)$ with $\|\underline{u}\|_{2, \af} = \sum_{i = 1}^3 \|u_i\|_{2, \af} < \frac{\dt}{10}$.
    \item Each $\bM_i$ has no self-intersection.
    \item \label{C^0_condition} ($\CC^0$-compatibility) For each $i \neq j$, $\bM_i$ only meets $\bM_j$ on the spine $\bigcap_{l = 1}^3 \bM_l$ of $\bM$. Moreover, the spine is still a circle of triple junctions.
\end{enumerate}
\end{defn}

\begin{lem} \label{lem_C^0}
    If $\underline{u} \in \underline{\CC}^{2, \af}([0, 1] \times \BS^1)$ satisfies $\|\underline{u}\|_{2, \af} < \frac{\dt}{10}$ and $\sum_{i = 1}^3 u_i(0, y) = 0$ for all $y \in \BS^1$, then there exists $\bv: \BS^1 \rightarrow \BR^2$ such that $$u_i(0, y) = \iip{\bv(y)}{\nu_i}, \quad \bw_i(y) = \iip{\bv(y)}{\bn_i} \bn_i$$ for all $i = 1, 2, 3$ and $y \in \BS^1$. Moreover, $\bM$ is a $\CC^{2, \af}$-perturbation of $(\bY \times \BS^1) \cap \bB$ associated with $(\underline{u}, \dt)$. Its spine is characterized by $$\bigcap_{i = 1}^3 \bM_i = \{(\bv(y), y) \, \big| \, y \in \BS^1 \}.$$
\end{lem}

\begin{proof}
    We first claim that the existence of $\bv$ is equivalent to $\sum_{i = 1}^3 u_i(0, y) = 0$ for all $y \in \BS^1$. If such $\bv$ exists, then $$\sum_{i = 1}^3 u_{i}(0, y) = \sum_{i = 1}^3 \iip{\bv(y)}{\nu_i} = \iip{\bv(y)}{\sum_{i = 1}^3 \nu_i} = 0.$$ Conversely, we suppose that $\sum_{i = 1}^3 u_i(0, y) = 0$. For all $i = 1, 2, 3$ and $y \in \BS^1$, we define $$\bv_i = \bw_i + u_i(0 ,y) \nu_i.$$ It suffices to show that $\bv_i = \bv_j$ for all $i, j$. According to \eqref{conormal} and \eqref{normal}, we note that $\iip{\bn_i}{\bn_{i-1}} = \iip{\bn_i}{\bn_{i+1}} = \iip{\nu_i}{\nu_{i-1}} = \iip{\nu_i}{\nu_{i+1}} = -\frac{1}{2}$ and $\iip{\bn_i}{\nu_{i-1}} = - \iip{\bn_i}{\nu_{i+1}} = \frac{\sqrt{3}}{2}$. Then $$\begin{aligned}
        \bv_i &= \frac{u_{i - 1}(0, y) - u_{i + 1}(0, y)}{\sqrt{3}} \bn_i + u_i(0, y) \nu_i \\
        &= \frac{u_{i - 1}(0, y) - u_{i + 1}(0, y)}{\sqrt{3}} (-\frac{1}{2} \bn_{i-1} + \frac{\sqrt{3}}{2} \nu_{i-1}) + u_i(0, y) (-\frac{\sqrt{3}}{2} \bn_{i-1} - \frac{1}{2} \nu_{i-1}) \\
        &= \frac{u_{i+1}(0, y) - u_{i-1}(0, y) - 3 u_i(0, y)}{2 \sqrt{3}} \bn_{i-1} + \frac{-u_{i + 1}(0, y) + u_{i-1}(0 ,y) - u_i(0, y)}{2} \nu_{i - 1} \\
        &= \frac{u_{i+1}(0, y) - u_i(0, y)}{\sqrt{3}} \bn_{i-1} + u_{i-1}(0 ,y) \nu_{i - 1} \\
        &= \bv_{i-1}.
    \end{aligned}$$ This holds for all $i$, which finishes the proof of the claim.
    
    Next, we claim that if $\|\underline{u}\|_{2, \af} < \frac{\dt}{10}$, then $\bM_i$ can not self-intersect. Note that $\|u_i\|_{0} \leq \|\underline{u}\|_{2, \af} < \frac{\dt}{10}$ for all $i = 1, 2, 3$. Then by the choice of $\eta$, we have $$ - 1 + \eta'(x) \frac{u_{i-1}(0, y)- u_{i+1}(0, y)}{\sqrt{3}} < -1 + \frac{2}{\dt} \cdot \frac{\|u_{i-1}\|_{0} + \|u_{i+1}\|_{0}}{\sqrt{3}} < 0,$$ which implies that $-x + \eta(x)\frac{u_{i-1}(0, y)- u_{i+1}(0, y)}{\sqrt{3}}$ is strictly decreasing in $x$ for any fixed $y \in \BS^1$. Therefore, $$\begin{aligned}
        \bM_i &= \{ (- x \, \bn_i + u_i(x, y) \nu_i + \eta(x) \bw_i(y), y) \, \big| \, (x, y) \in [0, 1] \times \BS^1 \} \\
        &= \{ ((- x + \eta(x)\frac{u_{i-1}(0, y)- u_{i+1}(0, y)}{\sqrt{3}}) \bn_i + u_i(x, y) \nu_i, y) \, \big| \, (x, y) \in [0, 1] \times \BS^1 \}
    \end{aligned}$$ can not self-intersect.

    According to the first claim, it is easy to see that $\{(\bv(y), y) \, \big| \, y \in \BS^1\} \subset \bigcap_{i = 1}^3 \bM_i$. We now claim that for each $i \neq j$, $\bM_i$ only meets $\bM_j$ at $\{(\bv(y), y) \, \big| \, y \in \BS^1\}$. This also implies $\bigcap_{i = 1}^3 \bM_i = \{(\bv(y), y) \, \big| \, y \in \BS^1\}$ and finishes the proof. It suffices to show that for any fixed $i = 1, 2, 3$ and $y \in \BS^1$, if \begin{equation} \label{eq1}
        - x\, \bn_i + u_i(x, y) \nu_i + \eta(x) \bw_i(y) = - z\, \bn_{i-1} + u_{i-1}(z, y) \nu_{i-1} + \eta(z) \bw_{i-1}(y),
    \end{equation} then $x = z = 0$.

    Writing $\bn_{i - 1}, \nu_{i -1}$ as linear combinations of $\bn_i$ and $\nu_i$, \eqref{eq1} can be rewritten as \begin{equation} \label{eq2} \begin{aligned}
        (\eta(x) - \eta(z)) \bv(y) &= (x + \frac{z}{2} + \frac{\sqrt{3}}{2}( u_{i-1}(z, y) - \eta(z) u_{i-1}(0, y))) \bn_i \\
        &\quad + (\frac{\sqrt{3} z}{2} -(u_i(x, y) - \eta(x) u_i(0, y)) - \frac{1}{2}( u_{i-1}(z, y) - \eta(z) u_{i-1}(0, y))) \nu_i.
    \end{aligned}
    \end{equation} Note that $\|\underline{u}\|_{2, \af} < \frac{\dt}{10}$ leads to \begin{equation} \label{sup_v}
        \| \bv \|_0 = \| \bv_i \|_0 = \sup_{y \in \BS^1} \sqrt{\left| \bw_i(y) \right|^2 + |u_i(0, y)|^2} \leq \sqrt{(\frac{\|u_{i-1}\|_0 + \|u_{i+1}\|_0}{\sqrt{3}})^2 + \|u_{i}\|_0^2} < \frac{\dt}{5}.
    \end{equation} Thus, the left hand side of \eqref{eq2} has length at most $\frac{\dt}{5}$. On the other hand, the right hand side of \eqref{eq2} has length at least $$|x + \frac{z}{2} + \frac{\sqrt{3}}{2}( u_{i-1}(z, y) - \eta(z) u_{i-1}(0, y))| \geq x + \frac{z}{2} - \frac{\sqrt{3} \dt}{10}.$$ Therefore, \eqref{eq2} holds only when $x + \frac{z}{2} - \frac{\sqrt{3} \dt}{20} < \frac{\dt}{5}$. In other words, $x, z \in [0, \dt]$. It follows that $\eta(x) = \eta(z) = 1$ and \eqref{eq2} becomes $$\begin{aligned}
        0 &= (x + \frac{z}{2} + \frac{\sqrt{3}}{2}( u_{i-1}(z, y) - u_{i-1}(0, y)) \bn_i \\
        &\quad+ (\frac{\sqrt{3} z}{2} -(u_i(x, y) -u_i(0, y)) - \frac{1}{2}( u_{i-1}(z, y) - u_{i-1}(0, y)) \nu_i.
    \end{aligned}$$ Now, by the mean value theorem and $\|u_{i-1}\|_{1} \leq \|\underline{u}\|_{2, \af} < \frac{\dt}{10}$, we have $$|u_{i-1}(z, y) - u_{i-1}(0, y)| \leq \frac{\dt}{10} z.$$ Then $$0 = |x + \frac{z}{2} + \frac{\sqrt{3}}{2}( u_{i-1}(z, y) - u_{i-1}(0, y))| \geq |x + \frac{z}{2} - \frac{\dt}{5} z|$$ leads to $x = z = 0$.
\end{proof}

To be a stationary varifold, $\bM$ must also satisfy the following two conditions: \begin{enumerate}
    \item \label{minimal} (Minimal condition) Each $\bM_i$, as a $\CC^2$ submanifold with boundary, is minimal in $\BR^2 \times \BS^1$. That is, the mean curvature vector $H_{\bM_i}$ in $\Int \bM_i$ is zero.
    \item \label{C^1_condition} ($\CC^1$-compatibility) The unit conormals $\xi_i \in \CC(\bigcap_{j = 1}^3 \bM_j, T(\BR^2 \times \BS^1)\big|_{\bigcap_{j = 1}^3 \bM_j}) \cong \CC(\BS^1, \BR^3)$ of the spine $\bigcap_{j = 1}^3 \bM_j$ in $\bM_i, i = 1, 2, 3$, should meet at $120^\circ$. In other words, \begin{equation}
        \sum_{i = 1}^3 \xi_i(y) = 0
    \end{equation} for all $y \in \BS^1$.
\end{enumerate} These two conditions follow from the first variation of the area functional $\CA$ and the divergence theorem: for any $\CC^1$-vector field $X$ on $\BR^2 \times \BS^1$ that vanishes on the boundary of $\bM$, \begin{equation*}
    \begin{aligned}
        0 &= \dt \CA (X) = \int_{\bM} div_{\bM}(X) \, d\CH^2 = \sum_{i = 1}^3 \int_{\bM_i} div_{\bM_i}(X^T + X^\perp) \, d\CH^2 \\
        &= \sum_{i = 1}^3 \int_{\bigcap^{3}_{j = 1} \bM_j} \iip{\xi_i}{X} \, d\CH^1 - \sum_{i = 1}^3 \int_{\bM_i} \iip{H_{\bM_i}}{X} \, d\CH^2
    \end{aligned}
\end{equation*} holds, where $\CH^k$ denotes the $k$-dimensional Hausdorff measure, $X^T$ and $X^\perp$ denote the tangential and normal part of $X$ to $\bM_i$ respectively.

The conditions \eqref{minimal} and \eqref{C^1_condition} come down to a system of PDEs.

\begin{prop} \label{prop_evo}
    Let $\underline{u}$ be as in Lemma \ref{lem_C^0}. Then $\bM$ is stationary if and only if $\underline{u}$ satisfies \begin{equation} \label{pde}
        \Dt \underline{u} = \underline{F}(\underline{u}) \coloneqq (F_1, F_2, F_3) \text{ in } (0, 1) \times \BS^1, \quad \SB \underline{u} = \underline{G}(\underline{u}) \coloneqq (0, G_1, G_2) \text{ on } \{0\} \times \BS^1, 
    \end{equation} where $\SB$ is the boundary operator \begin{equation}
        \SB \underline{u} = (u_1 + u_2 + u_3, \pl_\bn u_2 - \pl_\bn u_3, \pl_\bn u_1 - \frac{1}{2}(\pl_\bn u_2 + \pl_\bn u_3)),
    \end{equation} and $F_i$, $G_i$ have the following structure: 
    \begin{equation} \label{str_F&G}
    \begin{aligned}
        F_i(x, y) =& \, F_{i,1}(x, \underline{u}\big|_{(0, y)}, D_y \underline{u}\big|_{(0, y)}, D u_i\big|_{(x, y)}) \cdot D^2_{yy} \underline{u}\big|_{(0, y)} \\
        &+ F_{i, 2}(x, \underline{u}\big|_{(0, y)}, D_y \underline{u}\big|_{(0, y)}, D u_i\big|_{(x, y)}) \cdot D^2 u_i\big|_{(x, y)} \\
        &+ F_{i, 3}(x, \underline{u}\big|_{(0, y)}, D_y \underline{u}\big|_{(0, y)}, D u_i\big|_{(x, y)}), \\
        G_i(y) =& \, G_i(D \underline{u}\big|_{(0, y)}),
    \end{aligned}
\end{equation} where $F_{i, j}$, $G_i$ are smooth functions satisfying the structural condition: there exists $C > 0$ such that \begin{equation} \begin{aligned} \label{str_F}
        (|F_{i, 3}| + |D_{a_1} F_{i, 3}|)(a_1, a_2, a_3, a_4) &\leq C (|a_2| + |a_3| + |a_4|)^2, \\
        (\sum_{j = 1}^2 (|F_{i, j}| + |D_{a_1} F_{i, j}|) + \sum_{k = 2}^4 (|D_{a_k} F_{i, 3}| + |D^2_{a_k a_1} F_{i, 3}|))(a_1, a_2, a_3, a_4) &\leq C(|a_2|+ |a_3| + |a_4|), \\
        (\sum_{j = 1}^2 \sum_{k = 2}^4 (|D_{a_k} F_{i, j}| + |D^2_{a_k a_1 } F_{i, j}|) + \sum_{j = 1}^3 \sum_{2 \leq k \leq m \leq 4} |D^2_{a_k a_m} F_{i, j}|)(a_1, a_2, a_3, a_4) &\leq C, \\
    \end{aligned} 
    \end{equation} whenever $|a_1| \leq 1$ and $|a_2| + |a_3| + |a_4| \leq 1$; 
\begin{equation} \begin{aligned} \label{str_G}
        |G_i|(a) &\leq C |a|^2, \\
        |D_a G_i|(a) &\leq C |a|, \\
        (|D^2_{aa} G_i| + |D^3_{aaa} G_i|)(a) &\leq C,
    \end{aligned} 
    \end{equation} whenever $|a| \leq 1$.
\end{prop}

\begin{proof}
    We first introduce a terminology to shorten the discussion: $E$ is said to be a combination of $(a_1, a_2, \cdots, a_n)$ and $(b_1, b_2, \cdots, b_m)$ of class $(j, k)$ if it is of the form $$E = \sum_{\substack{j_1 + j_2 + \cdots + j_n = j \\ k_1 + k_2 + \cdots + k_m = k}} E_{j_1 j_2 \cdots j_n k_1 k_2 \cdots k_m} a_1^{j_1} a_2^{j_2} \cdots a_n^{j_n} b_1^{k_1} b_2^{k_2} \cdots b_m^{k_m},$$ where $E_{j_1 j_2 \cdots j_n k_1 k_2 \cdots k_m}$ are smooth functions of $a_1, a_2, \cdots, a_n, b_1, b_2, \cdots, b_m$. In particular, when $b_1 = \cdots = b_m = 1$, we simply call $E$ a combination of $a_1, \cdots, a_n$ of class $j$.
    
    Next, we interpret the minimal condition \eqref{minimal} as a PDE system. The calculations will be performed on a fixed $\bM_i$. For simplicity, let us omit the index $i$ from now on.

    On $\Int \bM_i$, there are two tangent vector fields induced by $\pl_x$ and $\pl_y$ respectively: \begin{equation}
        e_1 \coloneqq (-\bn + u_x \nu + \eta' \bw, 0) \quad \text{and} \quad e_2 \coloneqq ( u_y \nu + \eta \bw', 1).
    \end{equation} Then the induced metric $g$ and its inverse are given by \begin{equation*}
        g = \begin{pmatrix}
       \iip{e_1}{e_1} & \iip{e_1}{e_2} \\
       \iip{e_2}{e_1} & \iip{e_2}{e_2}
    \end{pmatrix} = \mathbb{1} + E_1, \quad g^{-1} = \mathbb{1} + E_2,
    \end{equation*} where $\mathbb{1}$ is the identity matrix, $E_1$ and $E_2$ are $2\times2$-matrices whose entries are combinations of $(u_x, u_y, \iip{\bw}{\bn}, \iip{\bw'}{\bn})$ and $(\eta, \eta')$ of class $(1, 0)$. In addition, the unit normal of $\bM_i$ in $\BR^2 \times \BS^1$ is of the form $$\Tilde{\nu} \coloneqq \frac{1}{\sqrt{1 + \bt^2 + \gm^2}} (\bt \bn + \nu, \gm),$$ where $\bt = \frac{u_x}{1 - \eta' \iip{\bw}{\bn}}$ and $\gm = - u_y - \bt \eta \iip{\bw'}{\bn}$. We then compute the second fundamental form $h$ of $\bM_i \subset \BR^2 \times \BS^1:$ $$\begin{aligned}
        h(e_1, e_1) &= \iip{\nabla_{e_1} e_1}{\Tilde{\nu}} = \iip{(u_{xx} \nu + \eta'' \bw, 0)}{\frac{(\bt \bn + \nu, \gm)}{\sqrt{1 + \bt^2 + \gm^2}}} = u_{xx} + E_3 u_{xx} + E_4, \\
        h(e_1, e_2) &= \iip{\nabla_{e_1} e_2}{\Tilde{\nu}} = \iip{(u_{xy} \nu + \eta' \bw', 0)}{\frac{(\bt \bn + \nu, \gm)}{\sqrt{1 + \bt^2 + \gm^2}}} = u_{xy} + E_5 u_{xy} + E_6, \\
        h(e_2, e_2) &= \iip{\nabla_{e_2} e_2}{\Tilde{\nu}} = \iip{(u_{yy} \nu + \eta \bw'', 0)}{\frac{(\bt \bn + \nu, \gm)}{\sqrt{1 + \bt^2 + \gm^2}}} = u_{yy} + E_7 u_{yy} + E_8 \iip{\bw''}{\bn},
    \end{aligned}$$ where $\nabla$ is the Levi-Civita connection on $(\BR^2 \times \BS^1, \iip{\cdot}{\cdot})$, $E_3, E_5, E_7, E_8$ are combinations of $(u_x, u_y, \iip{\bw}{\bn}, \iip{\bw'}{\bn})$ and $(\eta, \eta', \eta'')$ of class $(1, 0)$, and $E_4, E_6$ are combinations of $(u_x, u_y, \iip{\bw}{\bn}, \iip{\bw'}{\bn})$ and $(\eta, \eta', \eta'')$ of class $(2, 0)$. It follows that the mean curvature vector is of the form $$H_{\bM_i} = \tr(g^{-1}h) \Tilde{\nu} = (u_{xx} + u_{yy} + E_{9} D^2 u + E_{10} \iip{\bw''}{n} + E_{11}) \Tilde{\nu}, $$ where $E_{9}, E_{10}$ are combinations of $(D u, \iip{\bw}{\bn}, \iip{\bw'}{\bn})$ and $(\eta, \eta', \eta'')$ of class $(1, 0)$, and $E_{11}$ is a combination of $(D u, \iip{\bw}{\bn}, \iip{\bw'}{\bn})$ and $(\eta, \eta', \eta'')$ of class $(2, 0)$. Finally, we recall that $\iip{\bw_i(y)}{\bn_i} = \frac{u_{i-1}(0, y) - u_{i+1}(0, y)}{\sqrt{3}}$ to see that the minimal condition \eqref{minimal} is equivalent to that $u_i$ satisfies $$0 = \iip{H_{\bM_i}}{\Tilde{\nu}} = \Dt u_i - F_i$$ in $(0, 1) \times \BS^1$, where $F_i(x, y) = F_{i, 1} \cdot D^2_{yy} \underline{u} \big|_{(0, y)} + F_{i, 2} \cdot D^2 u_i \big|_{(x, y)} + F_{i, 3}$, and $F_{i, 1}, F_{i, 2}$ are combinations of $(\underline{u} \big|_{(0, y)}, D_y \underline{u}\big|_{(0, y)}, Du_i\big|_{(x, y)})$ and $(\eta(x), \eta'(x), \eta''(x))$ of class $(1, 0)$, $F_{i, 3}$ is a combination of $(\underline{u} \big|_{(0, y)}, D_y \underline{u}\big|_{(0, y)}, Du_i\big|_{(x, y)})$ and $(\eta(x), \eta'(x), \eta''(x))$ of class $(2, 0)$. Since $\eta$ has bounded $\CC^2$-norm, $F_{i, 1}, F_{i, 2}$ and $F_{i, 3}$ satisfy the desired structural condition \eqref{str_F}.

Now, we deal with the boundary operator $\SB$. The first term of $\SB\underline{u}$ is automatically $0$ by Lemma \ref{lem_C^0}. It remains to characterize the $\CC^1$-compatibility \eqref{C^1_condition}. It follows from Lemma \ref{lem_C^0} that the spine of $\bM$ is given by $y \mapsto (\bv(y), y)$, $y \in \BS^1$. Therefore, the tangent vector of the spine at $p = (\bv(y), y)$ is $(\bv'(y), 1)$. On the other hand, for fixed $y \in \BS^1$, $t \mapsto (-t \, \bn_i + u_i(t, y) \nu_i + \eta(t) \bw_i(y), y)$ is another curve on $M_i$ which touches the spine at $p$ when $t = 0$. Its tangent vector at $p$ is $(-\bn_i + D_x u_i \big|_{(0, y)} \nu_i, 0)$. Hence, the conormal of the spine in $\bM_i$ at $p$ is given by \begin{equation*}
    (-\bn_i + D_x u_i \big|_{(0, y)} \nu_i, 0) - \frac{\iip{(-\bn_i + D_x u_i \big|_{(0, y)} \nu_i, 0)}{(\bv'(y), 1)}}{|(\bv'(y), 1)|^2} (\bv'(y), 1),
\end{equation*} which is parallel to \begin{equation*}
    \begin{aligned}
        (-& (1 + |\bv'(y)|^2) \bn_i + (D_x u_i \big|_{(0, y)} + D_x u_i \big|_{(0, y)} |\bv'(y)|^2) \nu_i \\
        &+ (\iip{\bv'(y)}{\bn_i} - D_x u_i \big|_{(0, y)} \iip{\bv'(y)}{\nu_i}) \bv'(y), \iip{\bv'(y)}{\bn_i} - D_x u_i \big|_{(0, y)} \iip{\bv'(y)}{\nu_i}).
    \end{aligned}
\end{equation*} Recall that $\bv(y) = \frac{u_{i-1}(0, y) - u_{i+1}(0, y)}{\sqrt{3}} \bn_i + u_i(0, y) \nu_i$ by Lemma \ref{lem_C^0}. Then the unit conormal at $p$ is of the form \begin{equation*}
    \xi_i(y) = (- \bn_i + D_x u_i \big|_{(0, y)} \nu_i, \iip{\bv'(y)}{\bn_i}) + (H_{i, 1} \bn_i + H_{i, 2} \nu_i, H_{i, 3}),
\end{equation*} where $H_{i, 1}, H_{i, 2}, H_{i, 3}$ are combinations of $D\underline{u}\big|_{(0, y)}$ of class $2$. The $\CC^1$-compatibility \eqref{C^1_condition} therefore becomes \begin{equation} \label{C^1_eq} \begin{aligned}
    0 &= \sum_{i = 1}^3 \xi_i(y) \\
    &= (- \sum_{i = 1}^3 \bn_i + \sum_{i = 1}^3 D_x u_i \big|_{(0, y)} \nu_i, \iip{\bv'(y)}{\sum_{i = 1}^3 \bn_i}) + \sum_{i = 1}^3 (H_{i, 1} \bn_i + H_{i, 2} \nu_i, H_{i, 3}) \\
    &= (\sum_{i = 1}^3 D_x u_i \big|_{(0, y)} \nu_i + \sum_{i = 1}^3 H_{i, 1} \bn_i + \sum_{i = 1}^3 H_{i, 2} \nu_i, \sum_{i = 1}^3 H_{i, 3}).
\end{aligned}
\end{equation} We also notice that each $\xi(y)$ lives in the plane that is tangent to $(\bv'(y), 1)$. In other words, \begin{equation} \label{ogbasis}
    \{\xi(y)\}_{i = 1}^3 \in \spn\{(\bn_1, \frac{D_y u_2 \big|_{(0, y)} - D_y u_3 \big|_{(0, y)}}{\sqrt{3}}), (\nu_1, -D_y u_1 \big|_{(0, y)})\}
\end{equation} Recall that $\pl_{\bn} = - \pl_x$, where $\bn$ is the outer normal of $\{0\} \times \BS^1$ in $[0, 1] \times \BS^1 \subset \BR \times \BS^1$. By taking inner product of \eqref{C^1_eq} and vectors in \eqref{ogbasis} respectively, we conclude that the $\CC^1$-compatibility \eqref{C^1_condition} is equivalent to the following two scalar equations:
\begin{equation} \left\{\begin{aligned}
    (\pl_\bn u_2 - \pl_\bn u_3) \big|_{(0, y)} &= G_1\\
    (\pl_\bn u_1 - \frac{1}{2}(\pl_\bn u_2 + \pl_\bn u_3))\big|_{(0, y)} &= G_2,
\end{aligned}\right.
\end{equation} where $G_1, G_2$ are combinations of $D \underline{u} |_{(0, y)}$ of class $2$. This gives the desired structural condition \eqref{str_G} to the last two terms of $\SB\underline{u}$ and finishes the proof.
\end{proof}

\section{\bf The Linearization of (\ref{pde})}

Instead of directly dealing with the quasilinear PDE system \eqref{pde}, we first consider its linearization. In this section, we show the existence and uniqueness of the solution to the inhomogeneous system and derive a Schauder estimate. 

\begin{thm} \label{thm_lin}
    Given any $\underline{F} \in \underline{\CC}^{0, \af}([0, 1] \times \BS^1)$, $\underline{G} = (0, G_1, G_2) \in \underline{\CC}^{1, \af}(\BS^1)$ and $\underline{\vph} \in \underline{\CC}^{2, \af}(\BS^1)$, the inhomogeneous PDE system \begin{equation} \label{pde_linear}
        \left\{\begin{aligned}
            \Dt \underline{u} &= \underline{F} \text{ in } (0, 1) \times \BS^1 \\
            \SB \underline{u} &= \underline{G} \text{ on } \{0\} \times \BS^1 \\
            \underline{u} &= \underline{\vph} \text{ on } \{1\} \times \BS^1
        \end{aligned}\right.
    \end{equation} has a unique solution $\underline{u}_0 \in \underline{\CC}^{2, \af}([0, 1] \times \BS^1)$. Moreover, it has the following estimate: \begin{equation} \label{Schauder}
        \|\underline{u}_0\|_{2, \af} \leq C (\|\underline{F}\|_{0, \af} + \|\underline{G}\|_{1, \af} + \|\underline{\vph}\|_{2, \af}),
    \end{equation} where $C$ is a constant independent of $\underline{F}, \underline{G}$ and $\underline{\vph}$.
\end{thm}

Before tackling the complicated system \eqref{pde_linear}, we start with two simpler scalar-valued PDEs. One is the Dirichlet problem, and the other is the mixed boundary value problem.

\begin{prop} \label{prop_lin_Diri}
    Given any $f \in \CC^{0, \af}([0, 1] \times \BS^1)$ and $\vph \in \CC^{2, \af}(\BS^1)$, the inhomogeneous PDE \begin{equation} \label{pde_linear_Diri}
        \left\{\begin{aligned}
            \Dt v &= f \text{ in } (0, 1) \times \BS^1 \\
            v &= 0 \text{ on } \{0\} \times \BS^1 \\
            v &= \vph \text{ on } \{1\} \times \BS^1
        \end{aligned}\right.
    \end{equation} has a unique solution $v_0 \in \CC^{2, \af}([0, 1] \times \BS^1)$. Moreover, it has the following estimate: \begin{equation} \label{Schauder_Diri}
        \|v_0\|_{2, \af} \leq C (\|f\|_{0, \af} + \|\vph\|_{2, \af}),
    \end{equation} where $C$ is a constant independent of $f$ and $\vph$.
\end{prop}

\begin{proof}
    We first show the uniqueness. By subtracting the solutions, it suffices to show that \begin{equation*}
        \left\{\begin{aligned}
            \Dt v &= 0 \text{ in } (0, 1) \times \BS^1 \\
            v &= 0 \text{ on } (\{0\} \times \BS^1) \cup (\{1\} \times \BS^1)
        \end{aligned}\right.
    \end{equation*} has only zero solution. Indeed, Green's first identity leads to $$ 0 = \int_{(\{0\} \times \BS^1) \cup (\{1\} \times \BS^1)} v \pl_\bn v \, d \sm = \int_{(0, 1) \times \BS^1} v \Dt v \, dV + \int_{(0, 1) \times \BS^1} |\nabla v|^2 \, dV = \int_{(0, 1) \times \BS^1} |\nabla v|^2 dV.$$ Therefore, $\nabla v \equiv 0$ and $v \equiv 0$ by the boundary condition.

    Next, we derive the apriori estimate. Suppose that $v \in \CC^{2, \af}([0, 1] \times \BS^1)$. Since $[0, 1] \times \BS^1$ is compact, the classical interior (cf. \cite[Theorem 6.2]{GT01}) and boundary Schauder estimates (cf. \cite[Lemma 6.4, Theorem 6.6]{GT01}) yield \begin{equation} \label{Schauder_cl_Diri}
        \|v\|_{2, \af} \leq C (\|\Dt v\|_{0, \af} + \|v\big|_{\{0\}\times \BS^1}\|_{2, \af} + \|v\big|_{\{1\}\times \BS^1}\|_{2, \af} + \|v\|_0),
    \end{equation} where $C$ is a constant independent of $v$. We claim that \begin{equation} \label{ineq_max_Diri}
        \|v\|_0 \leq C (\|\Dt v\|_{0} + \sup_{(\{0\}\times \BS^1) \cup (\{1\}\times \BS^1)} |v|)
    \end{equation} for some constant $C$ independent of $v$. Let $F = \sup_{[0, 1] \times \BS^1} \max\{-\Dt v, 0\}$ and $\Phi = \sup_{(\{0\}\times \BS^1) \cup (\{1\}\times \BS^1)} \max \{v, 0\}$. Consider $$w(x, y) \coloneqq \Phi + (e - e^x) F \in \CC^2([0, 1] \times \BS^1)$$ to be the barrier. Note that $\Dt w = - e^x F \leq - F$ in $(0, 1) \times \BS^1$ and $v \leq \Phi$ in $[0, 1] \times \BS^1$. Then $\Dt (v-w) \geq \Dt v + F \geq 0$ in $(0, 1) \times \BS^1$ and $v - w \leq 0$ on $(\{0\}\times \BS^1) \cup (\{1\}\times \BS^1)$. Therefore, the maximum principle (cf. \cite[Theorem 3.73]{Aub82}) leads to $v - w \leq 0$ in $[0, 1] \times \BS^1$, which means that \begin{equation} \label{ineq_max_Diri_1}
        \sup_{[0, 1] \times \BS^1} v \leq \sup_{[0, 1] \times \BS^1} w \leq \Phi + e F \leq e(\sup_{[0, 1] \times \BS^1} \max\{-\Dt v, 0\} + \sup_{(\{0\}\times \BS^1) \cup (\{1\}\times \BS^1)} \max \{v, 0\}).
    \end{equation} A similar argument, by replacing $v$ with $-v$, shows that \begin{equation} \label{ineq_max_Diri_2}
        \sup_{[0, 1] \times \BS^1} - v \leq e(\sup_{[0, 1] \times \BS^1} \max\{\Dt v, 0\} + \sup_{(\{0\}\times \BS^1) \cup (\{1\}\times \BS^1)} \max \{-v, 0\}).
    \end{equation} \eqref{ineq_max_Diri} now follows from adding \eqref{ineq_max_Diri_1} and \eqref{ineq_max_Diri_2} together. Combining \eqref{Schauder_cl_Diri} with \eqref{ineq_max_Diri} gives \begin{equation} \label{apri_Diri}
         \|v\|_{2, \af} \leq C (\|\Dt v\|_{0, \af} + \|v\big|_{\{0\}\times \BS^1}\|_{2, \af} + \|v\big|_{\{1\}\times \BS^1}\|_{2, \af}),
    \end{equation} which is exactly \eqref{Schauder_Diri}.

    The last step is to prove the existence. We first show that \eqref{pde_linear_Diri} has a smooth solution if $f \in \CC^{\infty}([0, 1] \times \BS^1)$ and $\vph \in \CC^{\infty}(\BS^1)$. In this case, $f$ has a uniformly convergent Fourier series $$f(x, y) = f_0(x) + \sum_{k \in \BN} (f_{k, 1}(x) \cos(2 \pi k y) + f_{k, 2}(x) \sin(2 \pi k y)).$$ Moreover, for every $m_1 \in \BR$, $f_{k, j}$ has the decay $\|f_{k, j}\|_0 \leq C_{m_1} (1 + k)^{-m_1}$ for some constant $C_{m_1}$ independent of $k$ and $j$ (cf. \cite[Corollary 3.3.10 (b)]{Gra14}). Similarly, $\vph$ also has a uniformly convergent Fourier series $$\vph(y) = \vph_0 + \sum_{k \in \BN} (\vph_{k, 1} \cos(2 \pi k y) + \vph_{k, 2} \sin(2 \pi k y))$$ with decay $|\vph_{k, j}| \leq C_{m_2} (1 + k)^{-{m_2}}$ for all $m_2 \in \BR$. We now fix $m = m_1 + 2 = m_2$ and write $$v(x, y) = a_0(x) + \sum_{k \in \BN} (a_k(x) \cos(2 \pi k y) + b_k(x) \sin(2 \pi k y))$$ and try to solve $a_k, b_k$. Note that $$\Dt v = v_{xx} + v_{yy} = a_0''(x) + \sum_{k \in \BN} ((a_k''- 4 \pi^2 k^2 a_k)(x) \cos(2 \pi k y) + (b_k''- 4 \pi^2 k^2 b_k)(x) \sin(2 \pi k y)).$$ By comparing the coefficients, $a_0$ satisfies the second order ODE \begin{equation}
        \left\{\begin{aligned}
            &a_0''(x) = f_0(x) \quad \text{for } x \in (0, 1)\\
            &a_0(0) = 0 \\
            &a_0(1) = \vph_0,
        \end{aligned}\right.
    \end{equation} $a_k$ with $k \neq 0$ satisfies \begin{equation}
        \left\{\begin{aligned}
            &a_k''(x) - 4 \pi^2 k^2 a_k(x) = f_{k, 1}(x) \quad \text{for } x \in (0, 1)\\
            &a_k(0) = 0 \\
            &a_k(1) = \vph_{k ,1},
        \end{aligned}\right.
    \end{equation} and $b_k$ satisfies \begin{equation}
        \left\{\begin{aligned}
            &b_k''(x)- 4 \pi^2 k^2 b_k(x) = f_{k, 2}(x) \quad \text{for } x \in (0, 1)\\
            &b_k(0) = 0 \\
            &b_k(1) = \vph_{k ,2}.
        \end{aligned}\right.
    \end{equation} Therefore, $a_0$ is of the form \begin{equation} \label{a_0}
        a_0(x) = \int_0^x \int_0^s f_0(t) \, dt \, ds + (\vph_0 - \int_0^1 \int_0^s f_0(t) \, dt \, ds) x.
    \end{equation} On the other hand, it follows from the classical ODE theory that the general solution to $a_k''- 4 \pi^2 k^2 a_k = f_{k, 1}$ is of the form \begin{equation} \label{a_k}
        a_k(x) = (\int_1^x f_{k, 1}(t) e^{-2 \pi k t} \, dt + C_1) \frac{e^{2 \pi k x}}{4 \pi k} - (\int_0^x f_{k, 1}(t) e^{2 \pi k t} \, dt + C_2) \frac{e^{-2 \pi k x}}{4 \pi k}, 
    \end{equation} where $C_1, C_2$ are constants. Plugging in the boundary condition $a_k(0) = 0$ and $            a_k(1) = \vph_{k ,1}$, we see $$C_1 = \frac{4 \pi k e^{2 \pi k} \vph_{k, 1} - \int_0^1 f_{k, 1}(t) e^{-2 \pi k t} \, dt + \int_0^1 f_{k, 1}(t) e^{2 \pi k t} \, dt}{e^{4 \pi k} - 1}$$ and $$C_2 = \frac{4 \pi k e^{2 \pi k} \vph_{k, 1} - e^{4 \pi k} \int_0^1 f_{k, 1}(t) e^{-2 \pi k t} \, dt + \int_0^1 f_{k, 1}(t) e^{2 \pi k t} \, dt}{e^{4 \pi k} - 1}.$$ Similarly, \begin{equation} \label{b_k}
        b_k(x) = (\int_1^x f_{k, 2}(t) e^{-2 \pi k t} \, dt + \Tilde{C}_1) \frac{e^{2 \pi k x}}{4 \pi k} - (\int_0^x f_{k, 2}(t) e^{2 \pi k t} \, dt + \Tilde{C}_2) \frac{e^{-2 \pi k x}}{4 \pi k},
    \end{equation} where $$\Tilde{C}_1 = \frac{4 \pi k e^{2 \pi k} \vph_{k, 2} - \int_0^1 f_{k, 2}(t) e^{-2 \pi k t} \, dt + \int_0^1 f_{k, 2}(t) e^{2 \pi k t} \, dt}{e^{4 \pi k} - 1}$$ and $$\Tilde{C}_2 = \frac{4 \pi k e^{2 \pi k} \vph_{k, 2} - e^{4 \pi k} \int_0^1 f_{k, 2}(t) e^{-2 \pi k t} \, dt + \int_0^1 f_{k, 2}(t) e^{2 \pi k t} \, dt}{e^{4 \pi k} - 1}.$$ Note that \eqref{a_0}, \eqref{a_k}, \eqref{b_k} and the decay of $f_{k ,j}$, $\vph_{k, j}$ yield the decay of $a_k^{(n)}$, $b_k^{(n)}$: $$\|a_k^{(n)}\|_0 + \|b_k^{(n)}\|_0 \leq C (1 + k)^{-m + n},$$ where $C$ is some constant independent of $k$. Plancherel’s identity (cf. \cite[Proposition 3.2.7 (1)]{Gra14}) therefore shows that $v \in H^{m - \frac{1}{2} - \vep}([0, 1] \times \BS^1)$ for all $\vep > 0$. Then Sobolev embedding (cf. \cite[Theorem 2.30]{Aub82}) gives $v \in \CC^{m - \frac{3}{2} - \vep}([0, 1] \times \BS^1)$ for all $\vep > 0$. Let $m \rightarrow \infty$. We conclude that $v \in \CC^{\infty}([0, 1] \times \BS^1)$, and it is the solution, by the construction, to \eqref{pde_linear_Diri} whenever $f \in \CC^{\infty}([0, 1] \times \BS^1)$ and $\vph \in \CC^{\infty}(\BS^1)$.
    
    Finally, we consider the bounded linear operator $$\begin{aligned}
        \SL: \left\{ v \in \CC^{2, \af}([0, 1] \times \BS^1) \, \big| \, v \big|_{\{0\} \times \BS^1} = 0\right\} &\rightarrow \CC^{0, \af}([0, 1] \times \BS^1) \times \CC^{2, \af}(\BS^1) \\
        v &\mapsto (\Dt v, v \big|_{\{1\} \times \BS^1}).
    \end{aligned}$$ By the apriori estimate \eqref{apri_Diri}, $\SL$ has closed range. On the other hand, since $\CC^\infty([0, 1] \times \BS^1) \times \CC^\infty(\BS^1)$ is dense in $\CC^{0, \af}([0, 1] \times \BS^1) \times \CC^{2, \af}(\BS^1)$, the previous paragraph shows that the range of $\SL$ contains a dense subset. It follows that $\SL$ is surjective. This implies that \eqref{pde_linear_Diri} has a solution whenever $f \in \CC^{0, \af}([0, 1] \times \BS^1)$ and $\vph \in \CC^{2, \af}(\BS^1)$ and finishes the proof.
\end{proof}

\begin{prop} \label{prop_lin_mix}
    Given any $f \in \CC^{0, \af}([0, 1] \times \BS^1)$, $g \in \CC^{1, \af}(\BS^1)$ and $\vph \in \CC^{2, \af}(\BS^1)$, the inhomogeneous PDE \begin{equation} \label{pde_linear_mix}
        \left\{\begin{aligned}
            \Dt v &= f \text{ in } (0, 1) \times \BS^1 \\
            \pl_\bn v &= g \text{ on } \{0\} \times \BS^1 \\
            v &= \vph \text{ on } \{1\} \times \BS^1
        \end{aligned}\right.
    \end{equation} has a unique solution $v \in \CC^{2, \af}([0, 1] \times \BS^1)$. Moreover, it has the following estimate: \begin{equation} \label{Schauder_mix}
        \|v\|_{2, \af} \leq C (\|f\|_{0, \af} + \|g\|_{1, \af} + \|\vph\|_{2, \af}),
    \end{equation} where $C$ is a constant independent of $f$, $g$ and $\vph$.
\end{prop}

\begin{proof}
    We first show the uniqueness. By subtracting the solutions, it suffices to show that \begin{equation*}
        \left\{\begin{aligned}
            \Dt v &= 0 \text{ in } (0, 1) \times \BS^1 \\
            \pl_\bn v &= 0 \text{ on } \{0\} \times \BS^1 \\
            v &= 0 \text{ on } \{1\} \times \BS^1
        \end{aligned}\right.
    \end{equation*} has only zero solution. As in the proof of Proposition \ref{prop_lin_Diri}, Green's first identity yields $v \equiv 0$.

    Next, we derive the apriori estimate. Suppose that $v \in \CC^{2, \af}([0, 1] \times \BS^1)$. Due to the compactness of $[0, 1] \times \BS^1$ and that the Neumann boundary $\{0\} \times \BS^1$ and the Dirichlet boundary $\{1\} \times \BS^1$ are separated, the interior Schauder estimate (cf. \cite[Theorem 6.2]{GT01}), the boundary Schauder estimate for Dirichlet problem (cf. \cite[Lemma 6.4, Theorem 6.6]{GT01}) and the boundary Schauder estimate for Neumann problem (\cite[Lemma 6.27]{GT01}) yield \begin{equation} \label{Schauder_cl_mix}
        \|v\|_{2, \af} \leq C (\|\Dt v\|_{0, \af} + \|\pl_\bn v\big|_{\{0\}\times \BS^1}\|_{1, \af} + \|v\big|_{\{1\}\times \BS^1}\|_{2, \af} + \|v\|_0),
    \end{equation} where $C$ is a constant independent of $v$. We claim that \begin{equation} \label{ineq_max_mix}
        \|v\|_0 \leq C (\|\Dt v\|_{0} + \sup_{\{0\}\times \BS^1} |\pl_\bn v| + \sup_{\{1\}\times \BS^1} |v|)
    \end{equation} for some constant $C$ independent of $v$. Suppose not, then there exists a sequence $\{v_i \in \CC^{2,\af}([0, 1] \times \BS^1)\}_{i = 1}^{\infty}$ such that $\|v_i\|_0 = 1$ but $\|\Dt v_i\|_{0} + \sup_{\{0\}\times \BS^1} |\pl_\bn v_i| + \sup_{\{1\}\times \BS^1} |v_i| \rightarrow 0$ as $i \rightarrow \infty$. By \eqref{Schauder_cl_mix}, $\{v_i\}$ is a bounded sequence in $\CC^{2,\af}([0, 1] \times \BS^1)$. Since $\CC^{2, \af}([0, 1] \times \BS^1)$ embeds in $\CC^2([0, 1] \times \BS^1)$ compactly by Arzel\`{a}-Ascoli theorem, after passing to a subsequence, we may assume that there exists $w \in \CC^{2}([0, 1] \times \BS^1)$ such that $v_i \rightarrow w$ as $i \rightarrow$ in the $\CC^2$-topology. Then $\Dt w = \lim_{i \rightarrow \infty} \Dt v_i = 0$, $\pl_\bn w \big|_{\{0\} \times \BS^1} = \lim_{i \rightarrow \infty} \pl_\bn {v_i}\big|_{\{0\} \times \BS^1} = 0$, and $w\big|_{\{1\} \times \BS^1} = \lim_{i \rightarrow \infty} {v_i}\big|_{\{1\} \times \BS^1} = 0$. Hence, by the arguments for uniqueness, $w = 0$. However, $\|w\|_0 = \lim_{i \rightarrow \infty} \|v_i\|_0 = 1$, which is a contradiction. We therefore show \eqref{ineq_max_mix}. Combining \eqref{Schauder_cl_mix} with \eqref{ineq_max_mix} leads to \begin{equation} \label{apri_mix}
         \|v\|_{2, \af} \leq C (\|\Dt v\|_{0, \af} + \|\pl_\bn v\big|_{\{0\}\times \BS^1}\|_{1, \af} + \|v\big|_{\{1\}\times \BS^1}\|_{2, \af}),
    \end{equation} which is exactly \eqref{Schauder_mix}.

    Finally, we show the existence. As in the proof of Proposition \ref{prop_lin_Diri}, we first assume that $f$, $g$, $\vph$ are smooth and use Fourier series to construct the smooth solution to \eqref{pde_linear_mix} directly. Again, $f$ has a uniformly convergent Fourier series $$f(x, y) = f_0(x) + \sum_{k \in \BN} (f_{k, 1}(x) \cos(2 \pi k y) + f_{k, 2}(x) \sin(2 \pi k y))$$ with decay $\|f_{k, j}\|_0 \leq C_{m_1} (1 + k)^{-m_1}$ for all $m_1 \in \BR$, $\vph$ has a uniformly convergent Fourier series $$\vph(y) = \vph_0 + \sum_{k \in \BN} (\vph_{k, 1} \cos(2 \pi k y) + \vph_{k, 2} \sin(2 \pi k y))$$ with decay $|\vph_{k, j}| \leq C_{m_2} (1 + k)^{-{m_2}}$ for all $m_2 \in \BR$, and $g$ has a uniformly convergent Fourier series $$g(y) = g_0 + \sum_{k \in \BN} (g_{k, 1} \cos(2 \pi k y) + g_{k, 2} \sin(2 \pi k y))$$ with decay $|g_{k, j}| \leq C_{m_3} (1 + k)^{-{m_3}}$ for all $m_3 \in \BR$. We now fix $m = m_1 + 2 = m_2 = m_3 + 1$ and write $$v(x, y) = a_0(x) + \sum_{k \in \BN} (a_k(x) \cos(2 \pi k y) + b_k(x) \sin(2 \pi k y)).$$ Since $\pl_\bn = - \pl_x$, $a_0$ satisfies the second order ODE \begin{equation}
        \left\{\begin{aligned}
            &a_0''(x) = f_0(x) \quad \text{for } x \in (0, 1)\\
            &a_0'(0) = - g_0 \\
            &a_0(1) = \vph_0,
        \end{aligned}\right.
    \end{equation} $a_k$ with $k \neq 0$ satisfies \begin{equation}
        \left\{\begin{aligned}
            &a_k''(x)- 4 \pi^2 k^2 a_k(x) = f_{k, 1}(x) \quad \text{for } x \in (0, 1)\\
            &a_k'(0) = - g_{k, 1} \\
            &a_k(1) = \vph_{k ,1},
        \end{aligned}\right.
    \end{equation} and $b_k$ satisfies \begin{equation}
        \left\{\begin{aligned}
            &b_k''(x)- 4 \pi^2 k^2 b_k(x) = f_{k, 2}(x) \quad \text{for } x \in (0, 1)\\
            &b_k'(0) = - g_{k, 2} \\
            &b_k(1) = \vph_{k ,2}.
        \end{aligned}\right.
    \end{equation} As in the proof of Proposition \ref{prop_lin_Diri}, \begin{equation} \label{a_0(mix)}
        a_0(x) = \int_0^x \int_0^s f_0(t) \, dt \, ds - g_0 x + \vph_0 - \int_0^1 \int_0^s f_0(t) \, dt \, ds + g_0;
    \end{equation} \begin{equation} \label{a_k(mix)}
        a_k(x) = (\int_1^x f_{k, 1}(t) e^{-2 \pi k t} \, dt + C_1) \frac{e^{2 \pi k x}}{4 \pi k} - (\int_0^x f_{k, 1}(t) e^{2 \pi k t} \, dt + C_2) \frac{e^{-2 \pi k x}}{4 \pi k}, 
    \end{equation} where $$C_1 = \frac{4 \pi k e^{2 \pi k} \vph_{k, 1} - 2 g_{k, 1} + \int_0^1 f_{k, 1}(t) e^{-2 \pi k t} \, dt + \int_0^1 f_{k, 1}(t) e^{2 \pi k t} \, dt}{e^{4 \pi k} + 1}$$ and $$C_2 = \frac{- 4 \pi k e^{2 \pi k} \vph_{k, 1} - 2 e^{4 \pi k} g_{k, 1} + e^{4 \pi k} \int_0^1 f_{k, 1}(t) e^{-2 \pi k t} \, dt - \int_0^1 f_{k, 1}(t) e^{2 \pi k t} \, dt}{e^{4 \pi k} + 1};$$ \begin{equation} \label{b_k(mix)}
        b_k(x) = (\int_1^x f_{k, 2}(t) e^{-2 \pi k t} \, dt + \Tilde{C}_1) \frac{e^{2 \pi k x}}{4 \pi k} - (\int_0^x f_{k, 2}(t) e^{2 \pi k t} \, dt + \Tilde{C}_2) \frac{e^{-2 \pi k x}}{4 \pi k}, 
    \end{equation} where $$\Tilde{C}_1 = \frac{4 \pi k e^{2 \pi k} \vph_{k, 2} - 2 g_{k, 2} + \int_0^1 f_{k, 2}(t) e^{-2 \pi k t} \, dt + \int_0^1 f_{k, 2}(t) e^{2 \pi k t} \, dt}{e^{4 \pi k} + 1}$$ and $$\Tilde{C}_2 = \frac{- 4 \pi k e^{2 \pi k} \vph_{k, 2} - 2 e^{4 \pi k} g_{k, 2} + e^{4 \pi k} \int_0^1 f_{k, 2}(t) e^{-2 \pi k t} \, dt - \int_0^1 f_{k, 2}(t) e^{2 \pi k t} \, dt}{e^{4 \pi k} + 1}.$$ Note that \eqref{a_0(mix)}, \eqref{a_k(mix)}, \eqref{b_k(mix)} and the decay of $f_{k ,j}$, $g_{k, j}$, $\vph_{k, j}$ yield the decay of $a_k^{(n)}$, $b_k^{(n)}$: $$\|a_k^{(n)}\|_0 + \|b_k^{(n)}\|_0 \leq C (1 + k)^{-m + n},$$ where $C$ is a constant independent of $k$. Therefore, the same argument as in the proof of Proposition \ref{prop_lin_Diri} shows that $v$ is a smooth solution to \eqref{pde_linear_mix} when $f$, $g$, $\vph$ are smooth. Now, we define the bounded linear operator $$\begin{aligned}
        \SL: \CC^{2, \af}([0, 1] \times \BS^1) &\rightarrow \CC^{0, \af}([0, 1] \times \BS^1) \times \CC^{1, \af}(\BS^1) \times \CC^{2, \af}(\BS^1) \\
        v &\mapsto (\Dt v, \pl_\bn v \big|_{\{0\} \times \BS^1}, v \big|_{\{1\} \times \BS^1}).
    \end{aligned}$$ The apriori estimate \eqref{apri_mix} and the solvability of \eqref{pde_linear_mix} when $f$, $g$, $\vph$ are smooth show that the range of $\SL$ is a closed dense subset of $\CC^{0, \af}([0, 1] \times \BS^1) \times \CC^{1, \af}(\BS^1) \times \CC^{2, \af}(\BS^1)$. In other words, $\SL$ is surjective. The solvability of \eqref{pde_linear_mix} when $f \in \CC^{0, \af}([0, 1] \times \BS^1)$, $g \in \CC^{1, \af}(\BS^1)$, $\vph \in \CC^{2, \af}(\BS^1)$ therefore follows.
\end{proof}

Now we are able to prove Theorem \ref{thm_lin}.

\begin{proof}[Proof of Theorem \ref{thm_lin}]
    Define $v_1 \coloneqq u_1 + u_2 + u_3$, $v_2 \coloneqq u_2 - u_3$ and $v_3 \coloneqq u_1 - \frac{1}{2} (u_2 + u_3)$. Then $v_1$ satisfies \eqref{pde_linear_Diri} with $f = F_1 + F_2 + F_3$ and $\vph = \vph_1 + \vph_2 + \vph_3$; $v_2$ satisfies \eqref{pde_linear_mix} with $f = F_1 - F_2$, $g = G_1$ and $\vph = \vph_2 - \vph_3$; $v_3$ satisfies \eqref{pde_linear_mix} with $f = F_1 - \frac{1}{2} (F_2 + F_3)$, $g = G_2$ and $\vph = \vph_1 - \frac{1}{2}(\vph_2 + \vph_3)$. It follows from Proposition \ref{prop_lin_Diri} and Proposition \ref{prop_lin_mix} that $v_1$, $v_2$ and $v_3$ are uniquely determined. Note that \begin{equation} \label{eq_u_v}
        u_1 = \frac{1}{3}(v_1 + 2 v_3), \quad u_2 = \frac{1}{3}(v_1 - v_3) + \frac{1}{2} v_2, \quad u_3 = \frac{1}{3}(v_1 - v_3) - \frac{1}{2} v_2.
    \end{equation} This gives the unique solution $\underline{u}_0 = (u_1, u_2, u_3) \in \underline{\CC}^{2, \af}([0, 1] \times \BS^1)$ to \eqref{pde_linear}. Furthermore, \eqref{eq_u_v}, \eqref{Schauder_Diri} and \eqref{Schauder_mix} lead to $$\begin{aligned}
        \|\underline{u}_0\|_{2, \af} &= \|u_1\|_{2, \af} + \|u_2\|_{2, \af} + \|u_3\|_{2, \af} \\
        &\leq C(\|v_1\|_{2, \af} + \|v_2\|_{2, \af} + \|v_3\|_{2, \af})\\
        &\leq C(\|F_1 + F_2 + F_3\|_{0, \af} + \|\vph_1 + \vph_2 + \vph_3\|_{2, \af} \\
        &\qquad + \|F_1 - F_2\|_{0, \af} + \|G_1\|_{1, \af} + \|\vph_2 - \vph_3\|_{2, \af} \\
        &\qquad + \|F_1 - \frac{1}{2} (F_2 + F_3)\|_{0, \af} + \|G_2\|_{1, \af} + \|\vph_1 - \frac{1}{2}(\vph_2 + \vph_3)\|_{2, \af})\\
        &\leq C(\|F_1\|_{0, \af} + \|F_2\|_{0, \af} + \|F_3\|_{0, \af} + \|G_1\|_{1, \af} + \|G_2\|_{1, \af} + \|\vph_1\|_{2, \af} + \|\vph_2\|_{2, \af} + \|\vph_3\|_{2, \af}) \\
        &= C(\|\underline{F}\|_{0, \af} + \|\underline{G}\|_{1, \af} + \|\underline{\vph}\|_{2, \af}),
    \end{aligned}$$ where $C$ is a constant that may vary from line to line but still remains independent of $u_i$, $v_i$, $F_i$, $G_i$ and $\vph_i$. This is exactly \eqref{Schauder}, which finishes the proof.
\end{proof}

\section{\bf (\ref{pde}) with Dirichlet boundary condition and the proof of Theorem \ref{thm_main}}

We impose the Dirichlet boundary condition $\underline{u} = \underline{\vph}$ on $\{1\} \times \BS^1$ to \eqref{pde}.

\begin{thm} \label{thm_nonlin}
    Let $\underline{F}, \underline{G}$ be as in Proposition \ref{prop_evo}. Then there exists $\Tilde{r} \in (0, 1]$ satisfying the following: given any $r \in (0, \Tilde{r})$, there exists $\vep(r) > 0$ such that for any $\underline{\vph} \in \underline{\CC}^{2, \af}(\BS^1)$ with $\|\underline{\vph}\|_{2, \af} < \vep$, the system \begin{equation} \label{pde_Diri}
        \left\{\begin{aligned} 
            \Dt \underline{u} &= \underline{F}(\underline{u}) \text{ in } (0, 1) \times \BS^1 \\
            \SB \underline{u} &= \underline{G}(\underline{u}) \text{ on } \{0\} \times \BS^1 \\
            \underline{u} &= \underline{\vph} \text{ on } \{1\} \times \BS^1
        \end{aligned}\right.
    \end{equation} has a solution $\underline{u}_0$ with $\|\underline{u}_0\|_{2, \af} < r$.
\end{thm}

\begin{proof}
    Define $\SA: \underline{\CC}^{2, \af}([0, 1] \times \BS^1) \rightarrow \underline{\CC}^{2, \af}([0, 1] \times \BS^1)$ by setting $\SA(\underline{u}) = \underline{w}$ as the solution to \begin{equation}
        \left\{\begin{aligned} 
            \Dt \underline{w} &= \underline{F}(\underline{u}) \text{ in } (0, 1) \times \BS^1 \\
            \SB \underline{w} &= \underline{G}(\underline{u}) \text{ on } \{0\} \times \BS^1 \\
            \underline{w} &= \underline{\vph} \text{ on } \{1\} \times \BS^1.
        \end{aligned}\right.
    \end{equation} The well-definedness of $\SA$ follows from \eqref{str_F&G} and Theorem \ref{thm_lin}. We claim that there exist $r \in (0, 1)$ and $\vep > 0$ such that when $\|\underline{\vph}\|_{2, \af} < \vep$, $\SA\big|_{B_r}: B_r \rightarrow B_r$ is a contraction mapping, where $B_r$ is the open ball of radius $r$ centered at $0$ in $\underline{\CC}^{2, \af}([0, 1] \times \BS^1)$.

    Suppose that $\underline{u}, \underline{v} \in B_r \subset \underline{\CC}^{2, \af}([0, 1] \times \BS^1)$ with $r < 1$ and $\|\underline{\vph}\|_{2, \af} < \vep$. Then by the structural conditions \eqref{str_F}, \eqref{str_G} and the Schauder estimate \eqref{Schauder}, we have $$\|\SA(\underline{u})\|_{2, \af} \leq C(\|\underline{F}(\underline{u})\|_{0, \af} + \|\underline{G}(\underline{u})\|_{1, \af} + \|\underline{\vph}\|_{2, \af}) \leq C_1 (\|\underline{u}\|_{2, \af}^2 + \|\underline{\vph}\|_{2, \af}) \leq C_1(r^2 + \vep),$$ where $C_1$ is a constant independent of $r$ and $\vep$. On the other hand, $\SA(\underline{u}) - \SA(\underline{v})$ is the solution to \begin{equation}
        \left\{\begin{aligned} 
            \Dt \underline{w} &= \underline{F}(\underline{u}) - \underline{F}(\underline{v}) \text{ in } (0, 1) \times \BS^1 \\
            \SB \underline{w} &= \underline{G}(\underline{u}) - \underline{G}(\underline{v}) \text{ on } \{0\} \times \BS^1 \\
            \underline{w} &= 0 \text{ on } \{1\} \times \BS^1.
            \end{aligned}\right.
    \end{equation} Therefore, by \eqref{str_F}, \eqref{str_G} and \eqref{Schauder} again, we see $$\begin{aligned}
        \|\SA(\underline{u}) - \SA(\underline{v})\|_{2, \af} &\leq C(\|\underline{F}(\underline{u}) - \underline{F}(\underline{v})\|_{0, \af} + \|\underline{G}(\underline{u}) - \underline{G}(\underline{v})\|_{1, \af}) \\
        &\leq C_2 \|\underline{u} - \underline{v}\|_{2, \af} (\|\underline{u}\|_{2, \af} + \|\underline{v}\|_{2, \af}) \\
        &\leq 2 r C_2 \|\underline{u} - \underline{v}\|_{2, \af},
    \end{aligned}$$ where $C_2$ is a constant independent of $r$ and $\vep$. Consequently, for any $r < \Tilde{r} \coloneqq \min\{\frac{1}{C_1}, \frac{1}{4C_2}, 1\}$ and $\vep < r (\frac{1}{C_1} - r)$, $$\left\{\begin{aligned}
        \|\SA(\underline{u})\|_{2, \af} &< r \\
        \|\SA(\underline{u}) - \SA(\underline{v})\|_{2, \af} &\leq \frac{1}{2} \|\underline{u} - \underline{v}\|_{2, \af}
    \end{aligned}\right.$$ holds whenever $\underline{u}, \underline{v} \in B_r$ and $\|\underline{\vph}\|_{2, \af} < \vep$. In other words, $\SA\big|_{B_r}: B_r \rightarrow B_r$ is a contraction mapping when $\|\underline{\vph}\|_{2, \af} < \vep$.

    Finally, by the Banach fixed-point theorem, there exists $\underline{u}_0 \in B_r \subset \underline{\CC}^{2, \af}([0, 1] \times \BS^1)$ such that $\SA(\underline{u}_0) = \underline{u}_0$, and hence $\underline{u}_0$ is a solution to \eqref{pde_Diri} with $\|\underline{u}_0\|_{2, \af} < r$. 
\end{proof}

Theorem \ref{thm_main} is now a corollary of Theorem \ref{thm_nonlin}.

\begin{proof}[Proof of Theorem \ref{thm_main}]
    Let $r < \min\{\Tilde{r}, \frac{\dt}{10}\}$, where $\Tilde{r}$ is as in Theorem \ref{thm_nonlin}. Then Theorem \ref{thm_nonlin} guarantees the existence of $\vep(r) > 0$ and a solution $\underline{u}$ to \eqref{pde_Diri} with $\|\underline{u}\|_{2, \af} < r < \frac{\dt}{10}$ whenever $\|\underline{\vph}\|_{2, \af} < \vep$. Finally, Lemma \ref{lem_C^0} and Proposition \ref{prop_evo} show that $\underline{u}$ gives the desired $\CC^{2, \af}$-perturbation.
\end{proof}


\bibliographystyle{abbrv}
\bibliography{references}

\end{sloppypar}
\end{document}